\documentclass[english,10pt,final]{article}
\pdfoutput=1
\usepackage[english]{babel}
\usepackage{amsthm,amsmath,amssymb,amscd,amsxtra,amsgen,enumitem,verbatim}
\usepackage[?]{amsrefs}
\usepackage[all,cmtip]{xy}
\usepackage[utf8]{inputenc}

\setenumerate[1]{label=(\theenumi)}
{\bf}{\it}
\CompileMatrices
\begin{document}

\newcommand{\qr}[1]{\eqref{#1}} 
\newcommand{\var}[2]{\operatorname{var}_{#1}^{#2}}
\newcommand{\secref}[1]{\ref{#1}}
\newcommand{\varn}[1]{\var{#1}{}}

\newcommand{\ie}{{\it i.e.}\xspace} 
\newcommand{\eg}{{\it e.g.}\xspace}
\newcommand{\nc}{\newcommand}
\renewcommand{\frak}{\mathfrak}
\providecommand{\cal}{\mathcal}

\renewcommand{\bold}{\mathbf}
%
%
%
%
\newcommand{\gparens}[3]{{\left#1 #2 \right#3}}
\newcommand{\parens}[1]{\gparens({#1})} 
\newcommand{\brackets}[1]{\gparens\{{#1}\}} 
\newcommand{\hakparens}[1]{\gparens[{#1}]} 
\newcommand{\angleparens}[1]{\gparens\langle{#1}\rangle} 
\newcommand{\floor}[1]{\gparens\lfloor{#1}\rfloor} 
\newcommand{\ceil}[1]{\gparens\lceil{#1}\rceil}
\newcommand*{\Setof}[1]{\,\brackets{#1}\,} 
\newcommand*{\norm}[2][]{\gparens\|{#2}\|#1}
\newcommand*{\normsq}[2][]{\gparens\|{#2}\|^2#1}
%

\numberwithin{equation}{section}

\newcommand{\theoremname}{Theorem.}
\newcommand{\corollaryname}{Corollary.}
\newcommand{\lemmaname}{Lemma.}
\newcommand{\propositionname}{Proposition.}
\newcommand{\conjecturename}{Conjecture.}
\newcommand{\definitionname}{Definition.}
\newcommand{\examplename}{Example.}
\newcommand{\remarkname}{Remark.}
\newcommand{\pfname}{Proof.}

\newenvironment{pf}{\vskip-\lastskip\vskip\medskipamount{\it\pfname}}%
                      {$\square$\vskip\medskipamount\par}

\newenvironment{pfof}[1]{\vskip-\lastskip\vskip\medskipamount{\it
    Proof of #1.}}%
                      {$\square$\vskip\medskipamount\par}

\newtheorem{thm}{Theorem}[section]
\newtheorem{theorem}[thm]{Theorem}

\newtheorem{thmnn}{Theorem}
\renewcommand{\thmnn}{}

\newtheorem{cor}[thm]{Corollary}
\newtheorem{corollary}[thm]{Corollary}
\newtheorem{cornn}{Corollary}
 \renewcommand{\thecornn}{}

\newtheorem{prop}[thm]{Proposition}
\newtheorem{proposition}[thm]{Proposition}

\newtheorem{propnn}{Proposition}
\renewcommand{\thepropnn}{}

\newtheorem{lemma}[thm]{Lemma} 
\newtheorem{lemmat}[thm]{Lemma}
\newtheorem{lemmann}{Lemma}
\renewcommand{\thelemmann}{}

\theoremstyle{definition}
\newtheorem{defn}[thm]{Definition}
\newtheorem{definition}[thm]{Definition}
\newtheorem{defnn}{Definition}
\renewcommand{\thedefnn}{}
\newtheorem{conj}[thm]{Conjecture}
\newtheorem{axiom}{Axiom}

\theoremstyle{definition}
\newtheorem{exercise}{Exercise}[subsection]
\newtheorem{exercisenn}{Exercise}
\renewcommand{\theexercisenn}{}
\newtheorem{remark}[thm]{Remark}
\newtheorem{remarks}[thm]{Remarks}
\newtheorem{remarknn}{Remark}
\renewcommand{\theremarknn}{}
\newtheorem{example}[thm]{Example} 
\newtheorem{examples}[thm]{Examples} 
\newtheorem{examplenn}{Example}[subsection]
\newtheorem{blanktheorem}{}[subsection]
\newtheorem{question}[thm]{Question}
\renewcommand{\theexamplenn}{}


\nc{\Theorem}[1]{Theorem~{#1}}
\nc{\Th}[1]{({\sl Th.}~#1)}
\nc{\Thd}[2]{({\sl Th.}~{#1} {#2})}
\nc{\Theorems}[2]{Theorems~{#1} and ~{#2}}
\nc{\Thms}[2]{({\it Thms. ~{#1} and ~{#2}})}
\nc{\Lemmas}[2]{Lemma~{#1} and ~{#2}}

\nc{\manga}[6]{({\it Thms. ~ #1, ~ #2, ~ #3,\\ ~ #4, ~ #5, ~ #6})}

\nc{\Prop}[1]{({\sl Prop.}~{#1})}
\nc{\Proposition}[1]{Proposition~{#1}}
\nc{\Propositions}[2]{Propositions~{#1} and ~{#2}}
\nc{\Props}[2]{({\sl Props.}~{#1} and ~{#1})}
\nc{\Cor}[1]{({\sl Cor.}~{#1})}
\nc{\Corollary}[1]{Corollary~{#1}}
\nc{\Corollaries}[2]{Corollaries~{#1} and ~{#2}}
\nc{\Definition}[1]{Definition~{#1}}
\nc{\Defn}[1]{({\sl Def.}~{#1})}
\nc{\Lemma}[1]{Lemma~{#1}} 
\nc{\Lem}[1]{({\sl Lem.} ~{#1})} 
\nc{\Eq}[1]{equation~({#1})}
\nc{\Equation}[1]{Equation~({#1})}
\nc{\Section}[1]{Section~{#1}}
\nc{\Sections}[1]{Sections~{#1}}
\nc{\Sec}[1]{({\sl Sec.} ~{#1})} 
\nc{\Chapter}[1]{Chapter~{#1}}
\nc{\Chapt}[1]{({\sl Ch.}~{#1})}

\nc{\Ex}[1]{{\sl Ex.}~{#1}}
\nc{\Exa}[1]{{\sl Example}~{#1}}
\nc{\Example}[1]{{\sl Example}~{#1}}
\nc{\Examples}[1]{{\sl Examples}~{#1}}
\nc{\Exercise}[1]{{\sl Exercise}~{#1}}

\nc{\Rem}[1]{({\sl Rem.}~{#1})}
\nc{\Remark}[1]{{\sl Remark}~{#1}}
\nc{\Remarks}[1]{{\sl Remarks}~{#1}}
\nc{\Note}[1]{{\sl Note}~{#1}}

\nc{\Conjecture}[1]{Conjecture~{#1}}
\nc{\Claim}[1]{Claim~{#1}}

\nc \Proof{{  \it Proof. }}

\nc{\xmu}{\mu}
\nc{\w}{\omega}
\nc{\xv}{\mbox{\boldmath$x$}}
\nc{\uv}{\mbox{\boldmath$u$}}
\nc{\xiv}{\mbox{\boldmath$\xi$}}
\nc{\bbeta}{\mbox{\boldmath$\beta$}}
\nc{\balpha}{\mbox{\boldmath$\alpha$}}
\nc{\bgamma}{\mbox{\boldmath$\gamma$}}
\nc{\bdelta}{\mbox{\boldmath$\delta$}}
\nc{\bepsilon}{\mbox{\boldmath$\epsilon$}}

\newcommand{\ZZ}{{\mathbb Z}}
\newcommand{\RR}{{\mathbb R}} 

\nc \Ab{{\ensuremath{\bold A}}}
\nc \ab{{\ensuremath{\bold a}}}
\nc \bb{{\ensuremath{\bold b}}}
\nc \cb{{\ensuremath{\bold c}}}
\nc \Bb{{\ensuremath{\bold B}}}
\nc \Gb{{\ensuremath{\bold G}}}
\nc \Qb{{\ensuremath{\bold Q}}}
\nc \Rb{{\ensuremath{\bold R}}} \nc \Cb{{\ensuremath{\bold C}}} 
\nc \Eb{{\ensuremath{\bold E}}}
\nc \eb{{\ensuremath{\bold e}}}
\nc \Db{{\ensuremath{\bold D}}}
\nc \Fb{{\ensuremath{\bold F}}}
\nc \ib{{\ensuremath{\bold i}}}
\nc \jb{{\ensuremath{\bold j}}}
\nc \kb{{\ensuremath{\bold k}}}
\nc \nb{{\ensuremath{\bold n}}}
\nc \rb{{\ensuremath{\bold r}}}
\nc \Pb{{\ensuremath{\bold P}}}
\nc \pb{{\ensuremath{\bold p}}}
\nc \SPb{{\ensuremath{\bold {SP}}}}
\nc \Zb{{\ensuremath{\bold Z}}} 
\nc \zb{{\ensuremath{\bold z}}} 
\nc \gb{{\ensuremath{\bold g}}} 
\nc \fb{{\ensuremath{\bold f}}} 
\nc \ub{{\ensuremath{\bold u}}} 
\nc \vb{{\ensuremath{\bold v}}} 
\nc \yb{{\ensuremath{\bold y}}} 
\nc \xb{{\ensuremath{\bold x}}} 
\nc \xib{{\ensuremath{\bold \xi}}} 
\nc \Nb{{\ensuremath{\bold N}}} 
\nc \Hb{{\ensuremath{\bold H}}} 
\nc \wb{{\ensuremath{\bold w}}} 
\nc \Wb{{\ensuremath{\bold W}}} 
\nc \syz{{\mathbf {syz}}}
\nc \bnoll{{\ensuremath{\bold 0}}} 

\nc \mf{\frak m} \nc \mh{\hat{\m}} 
\nc \nf{\frak n}
\nc \Of{\frak O}
\nc \rf{\frak r}
\nc \mufr{{\mathbf \mu}}
\nc \hf{\frak h} 
\nc \qf{\frak q} 
\nc \bfr{\frak b} 
\nc \kfr{\frak k} 
\nc \pfr{\frak p} 
\nc \af{\frak a }
\nc \cf{\frak c }
\nc \sfr{\frak s} 
\nc \ufr{\frak u} 
\nc \g{\frak g} 
\nc \gA{\g_{\Ao}} 
\nc \lfr{\frak l}
\nc \afr{\frak a}
\nc \gfh{\hat {\frak g}}
\nc \gl{\frak { gl }}
\nc \Sl{\frak {sl}}
\nc \SU{\frak {SU}}
\nc{\Homf}{\frak{Hom}}

\newcommand{\on}{\operatorname}
\nc\hankel{\on {Hankel}}
\nc\row{\on {row\ }}
\nc\nullity{\on {nullity }}
\nc\col{\on {col\ }}
\nc\rowm{\on {Row \ }}
\nc\loc{\on {lc \ }}
\nc\nullo{\on {null\ }}
\nc\Nul{\on {Nul\ }}
\nc \Ann {\on {Ann }}
\nc \Ass {\on {Ass \ }}
\nc \Coker {\on {Coker}}
\nc \Co{\on C}
\nc \Homo{\on {Hom}}
\nc \Ker {\on {Ker}}
\nc \omod{\on {mod}}
\nc \No {\on N}
\nc \NN {\on {NN}}
\nc \NGo {\on {NG}}
\nc \Oo {\on O}
\nc \ch {\on {ch}}
\nc \rko {\on {rk}}
\nc \Sing {\on {Sing\ }}
\nc \Reg {\on {Reg}}
\nc \CoI {\on {CI}}
\nc \CoM {\on {CM}}
\nc \Gor {\on {Gor}}
\nc \Type {\on {Type}}
\nc \can {\on {can}}
\nc \Top {\on {T}}
\nc \Tr {\on {Tr}}
\nc \rel {\on {rel}}
\nc \tr {\on {tr}}
\nc \sgn {\on {sgn }}
\nc \trdeg {\on {tr.deg}}
\nc \codim {\on {codim }}
\nc \coht {\on {coht}}
\nc \divo {\on {div \ }}
\nc \coh {\on {coh}}
\nc \Clo {\on {Cl}}
\nc \embdim{\on {embdim}}
\nc \ed{\on {ed}}
\nc \embcodim{\on {embcodim  }}
\nc \qcoh {\on {qcoh}}
\nc \grad {\on {grad}\ }
\nc \grade {\on {grade}}
\nc \hto {\on {ht}}
\nc \depth {\on {depth}}
\nc \prof {\on {prof}}
\nc \reso{\on {res}}
\nc \ind{\on {ind}}
\nc \prodo{\on {prod}}
\nc \coind{\on {coind}}
\nc \Con{\on {Con}}
\nc \Crit{\on {Crit}}
\nc \Der{\on {Der}}
\nc \Char{\on {Char}}
\nc \Ch{\on {Ch}}

\nc \Ext{\on {Ext}}
\nc \Eo{\on {E}}
\nc \End{\on {End}}
\nc \ad{\on {ad}}
\nc \Ad{\on {Ad}}
\nc \gr{\on {gr}}
\nc \Fo{\on {F}}
\nc \Gr{\on {Gr}}
\nc \Go{\on {G}}
\nc \GFo{\on {GF}}
\nc \Glo{\on {Gl}}
\nc \PGlo{\on {PGl}}
\nc \Ho{\on {H}}
\nc \CMo{\on {\CM}}
\nc \SCM{\on {SCM}}
\nc \hol{\on {hol}}
\nc{\sgd}{\on{sgd}}
\nc \supp{\on {supp}}
\nc \ssupp{\on {s-supp}}
\nc \singsupp{\on {singsupp}}
\nc \msupp{\on {msupp}}
\nc \spec{\on {spec}}
\nc \spano{\on {span }}
\nc \Span{\on {Span }}
\nc \Max{\on {Max}}
\nc \Mat{\on {Mat}}
\nc \Min{\on {Min}}
\nc \nil{\on {nil}}
\nc \Mod{\on {Mod}}
\nc \Rad {\on {Rad}}
\nc \rad {\on {rad}}
\nc \rank {\on {rank}}
\nc \range {\on {range}}
\nc \Slo{\on {SL}}
\nc \soc {\on {soc}}
\nc \Irr {\on {Irr}}
\nc \Reo {\on {Re}}
\nc \Imo {\on {Im}}
\nc \SSo{\on {SS}}
\nc \lub{\on {lub}}
\nc \gldim{\on {gl.d.}}
\nc \length{\on {length}}
\nc \pdo{\on {p.d.}} 
\nc \fdo{\on {f.d.}} 
\nc \ido{\on {i.d.}} 
\nc \dSSo{\dot {\SSo}}
\nc \So{\on S}
\nc \Io{\on I}
\nc \Jo{\on J}
\nc \jo{\on j}
\nc \Ko{\on K}
\nc \PBW{\Ac_{PBW}}
\nc \Ro{\on R}
\nc \To{\on T}
\nc \Ao{\on A}

\nc \Do{{\on D}}
\nc \Bo{\on B}
\nc \Po{\on P}
\nc \Qo{\on Q}
\nc \Zo{\on Z}
\nc \U{\on U}
\nc \wt{\on {wt}}
\nc \Uh{\hat {\U}}
\nc \T{\on T}
\nc \Lo{\on L}
\nc{\dop}{\on d}
\nc{\eo}{\on e}
\nc{\ado}{\on{ad}}
\nc{\Tot}{\on{Tot}}
\nc{\Aut}{\on{Aut}}
\nc{\sinc}{\on {sinc}}

%
%
\nc{\overrightleftarrows}[2]{\overset{#1}{\underset{#2}{\rightleftarrows}}}

\nc{\CCF}{\cal{CF}}
\nc{\CDF}{\cal{DF}}
\nc{\CHC}{\check{\cal C}}

\nc{\Cone}{\on{Cone}}
\nc{\dec}{\on{dec}}
\nc{\Diff}{\on{Diff}}
\nc{\dirlim}{\underset{\to}{\on{lim}}}
\nc{\dpar}{\partial}
\nc{\GL}{\on{GL}}
\nc{\CGr}{\cal{G}r}
\nc{\pr}{\on{pr}}
\nc{\semid}{|\!\!\!\times}
\nc{\Hom}{\on{Hom}}
\nc \RHom{\on {RHom}}

\nc \Proj{\mathrm {Proj\ }}
\nc \proj{\mathrm {proj}}
\nc{\Id}{\on{Id}}
\nc{\id}{\on{id}}
\nc{\Ima}{\on{Im}}
\nc{\invtimes}{\underset{\gets}{\otimes}}
\nc{\invlim}{\underset{\gets}{\on{lim}}}
\nc{\Lie}{\on{Lie}}
\nc{\re}{\on{Re }}
\nc{\Pic}{\on{Pic }}
\nc{\LPic}{\on{LPic }}
\nc{\Sch}{\on{Sch}}
\nc{\Sh}{\on{Sh}}
\nc{\Set}{\on{Set}}
\nc{\spo}{\on{sp\  }}
\nc{\Spec}{\on{Spec}}
\nc{\mSpec}{\on{mSpec}}
\nc{\Specb}{\bold {Spec}\ }
\nc{\Projb}{\bold {Proj}}
\nc{\Specan}{\on{Specan}}
\nc{\Spo}{\on{Sp}}
\nc{\Spf}{\on{Spf}}
\nc{\sym}{\on{sym}}
\nc{\symm}{\on{symm}}
\nc{\rop}{\on{r}}
\nc{\Td}{\on{Td}}
\nc{\Tor}{\on{Tor}}


\nc{\Artin}{\cal{A}rtin}
\nc{\Dgcoalg}{\cal{D}gcoalg}
\nc{\Dglie}{\cal{D}glie}
\nc{\Ens}{\cal{E}ns}
\nc{\Fsch}{\cal{F}sch}
\nc{\Groupoids}{\cal{G}roupoids}
\nc{\Holie}{\cal{H}olie}
\nc{\Mor}{\cal{M}or}

\nc{\CF}{\ensuremath{\cal{F}}}
\nc \Kc{{\ensuremath{\cal K}}}
\nc \Lc{{\ensuremath{\cal L}}}
\nc \lcc{{\mathcal l}} 
\nc \CC{{\ensuremath{\cal C}}} 
\nc \Cc{{\ensuremath {\cal C}}}
\nc \Pc{{\ensuremath{\cal P}}}
\nc \Dc{\ensuremath{\mathcal D}}
\nc \Ac{{\ensuremath{\cal A}}} 
\nc \Bc{{\ensuremath{\cal B}}}
\nc \Ec{{\ensuremath{\cal E}}}
\nc \Fc{{\ensuremath{\cal F}}}
\nc \Mcc{{\ensuremath{\cal M}}} 
\nc \hM{\hat{\Mcc}} 
\nc \bM{\bar {\Mcc}} 
\nc\hbM{\hat{\bar \Mcc}}  
\nc \Nc{{\ensuremath{\cal N}}}
\nc \Hc{{\ensuremath{\cal H}}} 
\nc \Ic{{\ensuremath{\cal I}}} 
\nc \Oc{\ensuremath{{\cal O}}}
\nc \qc{\ensuremath{{\Cal q}}}
\nc \Och{\hat{\cal O}} 
\nc \Sc{{\ensuremath{{\cal S}}}}
\nc \Tc{\ensuremath{{\cal T}}} 
\nc \Vc{{\ensuremath{{\cal V}}}} 
\nc{\CA}{{\ensuremath{{\cal A}}}}
\nc{\CB}{{\ensuremath{{\cal B}}}}
\nc{\C}{{\ensuremath{{\cal F}}}}
\nc{\Gc}{{\ensuremath{{\cal G}}}}
\nc{\CH}{\ensuremath{\mathcal H}}
\nc{\CI}{{\ensuremath{{\cal I}}}}
\nc{\CM}{{\ensuremath{{\cal M}}}}
\nc{\CN}{{\ensuremath{{\cal N}}}}
\nc{\CO}{{\ensuremath{{\cal O}}}}
\nc{\Rc}{{\ensuremath{{\cal R}}}}
\nc{\CT}{{\ensuremath{\mathcal T}}}
\nc{\CU}{\ensuremath{{\cal U}}}
\nc{\CV}{\ensuremath{{\cal V}}}
\nc{\CZ}{\ensuremath{{\cal Z}}}
\nc{\Homc}{\ensuremath{{\cal {Hom}}}}


\nc{\fa}{\frak{a}}
\nc{\fA}{\frak{A}}
\nc{\fg}{\frak{g}}
\nc{\fh}{\frak{h}}
\nc{\fI}{\frak{I}}
\nc{\fK}{\frak{K}}
\nc{\fm}{\frak{m}}
\nc{\fP}{\frak{P}}
\nc{\fS}{\frak{S}}
\nc{\ft}{\frak{t}}
\nc{\fX}{\frak{X}}
\nc{\fY}{\frak{Y}}


\nc{\bF}{\bar{F}}
\nc{\bCP}{\bar{\cal{P}}}
\nc{\bm}{\mbox{\bf{m}}}
\nc{\bT}{\mbox{\bf{T}}}
\nc{\hB}{\hat{B}}
\nc{\hC}{\hat{C}}
\nc{\hP}{\hat{P}}
\nc{\htest}{\hat P}


\nc{\nen}{\newenvironment}
\nc{\ol}{\overline}
\nc{\ul}{\underline}
\nc{\ra}{\to}
\nc{\lla}{\longleftarrow}
\nc{\lra}{\longrightarrow}
\nc{\Lra}{\Longrightarrow}
\nc{\Lla}{\Longleftarrow}
\nc{\Llra}{\Longleftrightarrow}
\nc{\hra}{\hookrightarrow}
\nc{\iso}{\overset{\sim}{\lra}}

\nc{\dsize}{\displaystyle}
\nc{\sst}{\scriptstyle}
\nc{\tsize}{\textstyle}
\nen{exa}[1]{\label{#1}{\bf Example.\ } }{}


\nen{rem}[1]{\label{#1}{\em Remark.\ } }{}

\title{$\Dc$-modules with finite support are semi-simple }
\author{Rolf K\"allstr\"om}
\maketitle
\begin{abstract} Let $(R, \mf, k_R)$ be regular local $k$-algebra
  satisfying the weak Jacobian criterion, such that $k_R/k$ is an
  algebraic field extension. Let $\Dc_R$ be the ring of $k$-linear
  differential operators of $R$.  We give an explicit decomposition of
  the $\Dc_R$-module $\Dc_R/\Dc_R \mf_R^{n+1}$ as a direct sum of
  simple modules, all isomorphic to $\Dc_R/\Dc_R \mf$, where certain
  ``Pochhammer'' differential operators are used to describe
  generators of the simple components.
 \end{abstract}

\section{Introduction} 
The reason for this note is that J.-E. Bj\"ork teased me by asking
whether I knew if maximal ideals in a polynomial ring $A= k[x_1, \dots
, x_d]$ generate maximal left ideals in the Weyl algebra $\Dc_A$, also
when the ground field $k$ is not algebraically closed.  That $N=
\Dc_A/\Dc_A \mf_A$ is semisimple follows from Kashiwara's theorem that
the category of $\Dc_A$-modules with support in the point $\mf_A\in
\Spec A$ is equivalent to the category of vector space over the
residue field $k_A= A/\mf_A$, where the equivalence is $N \mapsto
N^{\mf_A}= \{n\in N \ | \ \mf \cdot n =0\}$ (see \cite [V.3.1.2,
VI.7.3]{borel:Dmod}; the argument works fine also when $k_A$ is not
algebraically closed).  In particular, $N$ will be simple if
$\dim_{k_A}N^{\mf_A}=1$. We will prove this, which surely is
well-known, but the main purpose of this note is to give explicit
semi-simple decompositions of $\Dc_R$-modules with a finite support,
where $\Dc_R$ is the ring of differential operators associated to a
rather general regular local ring $R$.

We shall work over a local regular noetherian $k$-algebra $(R, \mf,
k_R)$ of characteristic $0$, only requiring that the $R$-module of
$k$-linear derivations $T_{R/k}$ is big enough.

\begin{theorem}\label{goodring}(\cite{matsumura}*{Thms. 30.6, 30.8}) 
  Let $(R,\mf_R)$ be a regular local ring of dimension $n$ containing
  the rational numbers $\Qb$.  Let $R^*$  be a completion of $R$, $k_1$ a
  quasi-coefficient field of $R$, and $K$ be a coefficient field of
  $R^*$ such that $k_1\subset K$. The following conditions are
  equivalent:
  \begin{enumerate}
  \item There exist $\partial_1, \dots, \partial_n \in T_{R/k_1}$ and
    $f_1, \dots , f_n \in \mf_R$ such that $\det \partial_i(f_j)\not
    \in \mf_R$.
  \item If $\{x_1, \dots , x_n\}$ is a regular system of parameters
    and $\partial_{x_i}$ are the partial derivatives of $R^* = K[[x_1,
    \dots , x_n]]$, $\partial_{x_i}(x_j)= \delta_{ij}$, then
    $\partial_{x_i} \in T_{R/k_1}$.
  \item $T_{R/k_1}$ is free of rank $n$.  
  \end{enumerate}
  Furthermore, if these conditions hold, then for any $P\in \Spec R$,
  putting $A= R/P$, we have $T_{A/k_1}= T_{R/k_1}(P)/PT_{R}$, and
  $\rank T_{A/k_1} = \dim A$.  ($T_{R/k_1}(P)\subset T_{R/k_1}$
  denotes the submodule of derivations $\partial$ such that $\partial
  (P)\subset P$.)
\end{theorem}
If the equivalent conditions in \Theorem{\ref{goodring}} hold, then we
say that $(R, \mf, k_R)$ satisfies the weak Jacobian condition
$(WJ)_{k_1}$. Note that if $k_R/k$ is algebraic, then we can replace
$k_1$ by $k$ and write $(WJ)_k$.  In this paper $(R, \mf, k_R)$
denotes a $k$-algebra of characteristic $0$ satisfying $(WJ)_{k}$,
where the field extension $k_R/k$ is algebraic.

For example, $R$ could be the localisation at a maximal ideal of a
regular ring of finite type over $k$, a formal power series ring over
$k$, or a ring of convergent power series when $k$ is either the field
of real or complex numbers.

Recall that the ring of ($k$-linear) differential operators $\Dc_R 
\subset \End_k (R)$ of $R$ is defined inductively as $\Dc_R= \cup
_{m\geq 0} \Dc_R^m $, $\Dc^0= \End_R(R)= R$, $\Dc_R^{m+1}= \{P
\in \End_k (R) \ \vert \ [P, R]\subset \Dc_R^m\}$, where $[P, R]= PR -
R P \subset \End_k (R)$.  It is easy to see that $T_R\subset \Dc^1_R
\subset \Dc_R$, and conversely, if $P\in \Dc^1_R$, then $P - P(1) \in
T_R$; hence
\begin{displaymath}
  \Dc^1_R = R+ T_R .
\end{displaymath}

The following companion to \Theorem{\ref{goodring}} should be well
known; see \cite{kallstrom-tadesse:liehilbert}.

\begin{proposition}\label{diffop-gen-deg1}
  Let $R/k$ be a regular local $k$-algebra satisfying $(WJ)_k$ and
  such that $k_R/k$ is algebraic. Then $\Dc^1_R$ generates the algebra
  $\Dc_R$.
\end{proposition}
Select $x_i$ and $\partial_{x_i}$ as in
\Theorem{\ref{goodring}}. Given a multi-index $\alpha = (\alpha_1,
\dots , \alpha_n)$, $n= \dim R$, we put $X^\alpha =
x_1^{\alpha_1}x_2^{\alpha_2} \cdots x_n^{\alpha_n} \in R$,
$\partial^\alpha
= \partial_{x_1}^{\alpha_1} \partial_{x_2}^{\alpha_2}\cdots \partial_{x_n}^{\alpha_n}\in
\Dc_R$, $|\alpha| = \sum\alpha_i$, and $\alpha ! = \alpha_1!  \cdots
\alpha_n!$.  

We recall some important well-known facts for the algebra $\Dc_R$, which we later
on will refer to as (Facts):

\begin{enumerate}
\item the $R$-module $\Dc_R$ is free with basis
  $\{\partial^\alpha\}_{\alpha\in \Nb^d} $, where $\Dc_R$ is either
  regarded as left or right module.
\item $R$ is a simple $\Dc_R$-module.
\item $\Dc_R$ is a simple ring.
\end{enumerate}

\begin{proof}
  (1): That the $\partial^\alpha$ generate $\Dc_R$ both as left or
  right $R$-algebra follows from
  \Proposition{\ref{diffop-gen-deg1}}. First consider $\Dc_R$ as left
  $R$-module. Assume that $P=\sum_{\alpha \in \Omega}
  a_\alpha \partial^\alpha =0 $, where $\Omega$ is a finite set of
  multi-indices. If one of the indices has minimal $|\alpha |$ in the
  set $\Omega$, then $P(X^\alpha) = a_{\alpha}\alpha !  =0$. This
  implies that $a_\alpha =0$ for all $\alpha$. Now take the right
  module structure, and assume $\sum_{\alpha \in
    \Omega} \partial^\alpha a_\alpha =0$. Then $\sum_{\alpha \in
    \Omega} (a_\alpha \partial^\alpha + [\partial^\alpha, a_\alpha])
  =0$, where $[\partial^\alpha, a_\alpha] \in \Dc_R^{|\alpha|-1}$ and
  $a_\alpha \partial^\alpha \in \Dc_R^{|\alpha|}$. Since the
  $\partial^\alpha$ are free generators as left module, it follows
  that if $\alpha \in \Omega$ has maximal $|\alpha|$, then $a_\alpha
  =0$. This implies that all $a_{\alpha}=0$.  (2): Let $I\subset R$ be
  a non-zero $\Dc_R$-module. If $I\neq R$ there exists a non-zero
  element $f\in I\cap \mf^l$ of smallest $l \geq 1$. But then there
  exists a derivation $\partial$ such that $\partial (f)\in
  \mf^{l-1}$, $\partial (f)\neq 0 $, which gives a contradiction.
  (3): If $P\in \Dc^n_R $ belongs to a 2-sided ideal $J$, then
  $P_r=[r,P]\in J\cap \Dc^{n-1}_R $ for all $r\in R$. Unless $P
  \not\in \Dc_R^0= R$ there exists an element $r$ such that $P_r \neq
  0$.  Iterating, it follows that $J\cap R\neq 0$.  By (2) $R \subset
  J$; hence $J= \Dc_R$.
\end{proof}

\section{$\Dc$-modules with finite support}
Let $\Dc_X = \Dc_{X/k} $ denote the sheaf of differential operators on
a scheme $X/k$; we refer to \cite{EGA4:4} for the basic
definitions. Instead of schemes we could in a similar way consider
sheaves on complex or real analytic manifolds (or even ringed spaces
where the local rings are regular and satisfy $(WJ)_k$ at all closed
points), but the reader will have little problems in transcribing the
theorem below to such a situation.

The theorem below can be regarded as a version of Kashiwara's
embedding theorem: 
\begin{theorem}\label{semisimplefinitesupport}
  Let $X/k$ be a scheme of characteristic $0$ such that the local ring
  at all closed points are regular and satisfies $(WJ)_k$, and that
  all closed points are rational over $k$. Let $M$ be a coherent
  $\Dc_{X/k}$ -module whose support $\supp M \subset X$ is a finite
  set of closed points.  Let $n_x$ be the length of the maximal
  submodule of $M$ with support at the point $x$, and let $\mf_x$ be
  the sheaf of ideals of $x$.  Then $M$ is a semi-simple module of the
  form
  \begin{displaymath}
    M =    \bigoplus_{x\in \supp M} \bigoplus ^{n_x} \frac {\Dc_X}{\Dc_X \mf_x },
  \end{displaymath}
and $n_x= \dim_{k_x} M^{\mf_x}$. 
\end{theorem}
We remark that by (Fact 1) $\Dc_R/\Dc_R \mf = k_R[\partial_{x_1},
\dots , \partial_{x_d}]$, where the action of $k_R[x_1, \dots , x_d]$
on the right side is determined by $X^\alpha \partial^\beta = - \frac
{\beta !}{\alpha !} \partial^{\beta-\alpha}$, when $\beta_i \geq
\alpha_i$, $i=1, \dots , n$, and otherwise $X^\alpha \partial^\beta
=0$.  

Our goal is to give a concrete decomposition when we are given
a presentation in terms of cyclic modules.  First we have the following
well-known lemma:

\begin{lemma}(\cite{stafford:modulestructure})\label{cyclic} Let $\Dc$
  be a simple ring (it has no non-trivial 2-sided ideals) and $M$ be a
  $\Dc$-module of finite length. Assume that for any element $m$ in $
  M$ there exists a non-zero element $P$ in $ \Dc$ such that $Pm=0$
  ($M$ is a torsion module).  Then $M$ is cyclic.
\end{lemma}

We remark that artinian $\Dc$-modules are torsion in the above sense
if the ring $\Dc$ is not artinian. For example, $\Dc_R$ is not
artinian as soon as $\dim R \geq 1$, so in particular any
$\Dc_R$-module of finite length is cyclic.

If now $M$ is a $\Dc_R$-module of finite type with support at the
maximal ideal $\mf$, then $\dim_{k_R} M^\mf < \infty$, and any set of
generators of $M$ will belong to $M^{\mf^n}$ for sufficiently high
$n$. Therefore $M$ cannot have an infinite composition series,
i.e. $M$ is of finite length, and since any element in $M$ is killed
by $\mf^n$ for sufficiently high $n$, it is clearly a torsion module
(which we thus can see without using the fact that $\Dc_R$ is
non-artinian); hence $M$ is cyclic by \Lemma{\ref{cyclic}}. If $m$ is
a cyclic generator we have a surjective homomorphism $\Dc_R
/\Dc_R\mf^{n+1}\to \Dc_R m = M$, so that after iteration we have a
finite resolution
\begin{equation}\label{equation}
0 \to \frac{\Dc_R}{\Dc_R \mf^{n_r }} \to \cdots \to
\frac{\Dc_R}{\Dc_R \mf^{n_i} }\to \cdots \to \frac{\Dc_R}{\Dc_R \mf^{n+1}}\to M \to 0
.
\end{equation}

\begin{pfof}{\Lemma{\ref{cyclic}}}
  We prove this by induction over the length of a $\Dc$-module $M$. If
  $l(M)=1$, $M$ is simple so any non-zero vector is a cyclic
  generator. Now assume $l(M)\geq 2$ and that the assertion holds for
  all modules of length $<l(M)$. If $L\subset M$ is a non-zero simple
  submodule, we have the exact sequence
\begin{displaymath}
  0\to L \to M \to M/L \to 0
\end{displaymath}
where $l(M/L) < l(M)$. By assumption there exists an element $m$ in
$M$ that maps to a cyclic generator in $M/L$. Choose a non-zero vector
$m_0 \in L$.  Since $M$ is a torsion module, $\Ann_\Dc (m) \neq 0$,
and since $\Dc$ is simple, the 2-sided ideal $ \Ann_\Dc (m)\Dc$
contains the identity $1$; hence there exists $Q\in \Ann_{\Dc} (m)$
and $P\in \Dc$ such that $QP m_0 \neq 0$. Putting $m_1 = m + Pm_0$, we
have $Qm_1 = QPm_0 \in L$, and since $L$ is simple, both $Pm_0$ and $
m_0$ belong to $ \Dc m_1$; hence also $m\in \Dc m_1$.  By assumption
any element $m'$ in $ M$ can be written $m'= P_0m_0 +P_1m $; since $m,
m_0 \in \Dc m_1$ this shows that $m_1$ is a cyclic generator of $M$.
\end{pfof}

Recall  the Pochhammer symbol:
\begin{displaymath}
(a)_n = a(a+1)\cdots (a+n-1),
\end{displaymath}
and we also put $(a)_0 =1$. In the theorem below we use the notation
in \Theorem{\ref{goodring}}.
\begin{theorem}\label{decompDmod}
  Let $(R,\mf, k_R) $ be an allowed regular local $k$-algebra $R$ of
  dimension $d$, such that the residue field $k_R= R/\mf$ is algebraic
  over $k$, and let $\Dc_R$ be the ring of differential operators of
  $R$.  Define the derivations $\partial_{x_i}$ by
  $\partial_{x_i}(x_j) = \delta_{ij}$ and the ``Pochhammer''
  differential operators
  \begin{displaymath}
    Q_{n,d}(x_1, \dots , x_d) = \prod_{i=1}^d (1+ \partial_{x_i}x_i)_n
    \in \Dc_R.
  \end{displaymath}
  \begin{enumerate}
  \item $\Dc_R/\Dc_R \mf$ is a simple $\Dc_R$-module and $\Dc_R/\Dc_R \mf^{n+1}$
    is a semi-simple $\Dc_R$-module for each positive integer $n$.
\item   There is an isomorphism of $\Dc_R$-modules
\begin{eqnarray*}
  \psi:   \bigoplus_{j=0}^{n} \bigoplus_{|\alpha|=j} \frac {\Dc_R}{\Dc_R \mf} &\to&
  \frac {\Dc_R}{\Dc_R \mf^{n+1}}, \\
  (P_{\alpha,j}\omod \Dc_R \mf)& \mapsto & P_{\alpha,j } Q_{n-j,d}(x_1, \dots , x_d) X^{\alpha}
  \omod \Dc_R \mf^{n+1}
\end{eqnarray*}
  \end{enumerate}
\end{theorem}

\begin{lemma}\label{simplelemma}  Let $M$ be a  $\Dc_R$-module which is
  generated by its $\mf$ -invariant subspace $M^{\mf}= \{m \in M \
  \vert \ \mf\cdot m =0\}$.  Then $M$ is semi-simple. More precisely,
  if $S$ is a basis of the $k_R$-vector space $M^{\mf}$, we have
  \begin{displaymath}
    M= \bigoplus_{v\in S} \Dc_R v,
  \end{displaymath}
  where all the modules $\Dc_R v$ are isomorphic to the simple module $
  \Dc_R/\Dc_R \mf$.
\end{lemma}
\begin{proof}

  We first note that if $L$ is a $\Dc_R$-module of finite type which
  is generated by the invariant space $L^{\mf}$, and this space is
  one-dimensional over $k_R$, then $L$ is simple.  This follows since
  any element in $L$ is killed by a sufficiently high power of $\mf$,
  so if $L_1$ is a non-zero submodule we have $L_1^{\mf}\neq 0$. Hence
  $L_1^{\mf} = L^{\mf}$, which gives $L_1=L$.

  To see that the module $N= \Dc_R/\Dc_R \mf$ is simple, by the
  previous paragraph it suffices to prove that $N^{\mf}$ is
  one-dimensional over $k_R$.  So if $P\in \Dc_R$ and $ P \omod \Dc_R
  \mf \in N^{\mf}$, i.e.  $\mf P \subset \Dc_R\mf $, we need to see
  that $P \in R + \Dc_R \mf$.  Expressed in a regular system of
  parameters $P= \sum_{\alpha} \partial^\alpha a_\alpha $ we have
  $(x_1, \dots , x_d)\cdot \sum_{\alpha} \partial^\alpha a_\alpha
  \subset \Dc_R \cdot (x_1, \dots , x_d)$. This implies, from the fact
  that the differential operators $\partial^\alpha$ form a basis of
  the right $R$-module $\Dc_R$, that $a_\alpha \in (x_1, \dots , x_d)$
  when $|\alpha | >0 $, and therefore $P\in R + \Dc_R \mf$.

  If $v\in M^{\mf}$, then there is a canonical non-zero homomorphism $
  \Dc_R/\Dc_R \mf \to \Dc_R v$, which is injective and has a simple
  image by the previous paragraph.  The canonical surjective
  homomorphism
  \begin{displaymath}
    \bigoplus_{v\in S} \Dc_R v \to M
  \end{displaymath}
  is an isomorphism since the left hand side is semi-simple and the
  restriction to any of its simple terms is non-zero.
\end{proof}
\begin{pfof}{\Theorem{\ref{semisimplefinitesupport}} and
    \Theorem{\ref{decompDmod}}} The decomposition of $M$ over the
  support is obvious, so one can assume that $\supp M = \{x\}$ is a
  single closed point in $X$, and thus $M$ can be regarded as a
  $\Dc_R$-module, where $R$ is the local ring $\Oc_{X,x}$.  Let $M^0$
  be an $R$ -submodule of finite type that generates $M$, hence
  $\mf^{n+1} M^0 =0$ for high $n$; let $n$ be the highest integer such
  that $\mf^n M^0 \neq 0$.  Put $M^0_i = \mf^i M^0$ and let $M_i $ be
  the $\Dc_R$-module it generates, so we have a filtration by
  $\Dc_R$-modules $M_n \subset M_{n-1} \subset \cdots \subset M_0=M$,
  and exact sequences
  \begin{displaymath}
    0 \to     M_{i+1} \to M_{i} \to M_{i}/M_{i+1}\to 0,
  \end{displaymath}
  Then $M_n $ and each quotient $M_{i}/M_{i+1}$ is generated by its
  $\mf$-invariants, so at any rate $M$ is a successive extension of
  $\Dc_R$-modules that are generated by $\mf$-invariants. All these
  modules can be decomposed into a direct sum of simple modules that
  are isomorphic to the module $\Dc_R/\Dc_R \mf$.  It remains to see
  that $M$ is a direct sum of such modules. To see this first note
  that $M$ is a quotient of a direct sum of modules of the form
  $\Dc_R/\Dc_R\mf^{n+1}$ for different non-negative integers $n$, so
  to see that $M$ is semi-simple it suffices to see that $
  \Dc_R/\Dc_R\mf^{n+1}$ is semi-simple, and this follows if we prove
  (2) in \Theorem{\ref{decompDmod}}.

  We have $x_i^k (k+\partial_{x_i} x_i) = \partial_{x_i} x_i^{k+1} $,
  hence $x_i(1+\partial_{x_i}x_i)_{n-j} = \partial_{x_i}^{n-j}
  x_i^{n+1-j} $. Therefore $x_i Q_{n-j,d}(x_1, \dots , x_d) X^{\alpha}
  = Q_{n-j,d-1}(x_1, \dots, \hat x_i, \cdots , x_d)
  x_i^{n+1-j}X^\alpha $, so if $|\alpha| =j$, then
  \begin{displaymath}\tag{$*$}
    \mf Q_{n-j,d}(x_1, \dots , x_d) X^{\alpha} \subset \Dc_R \mf^{n+1}.
\end{displaymath}
 Therefore there exists a
  homomorphism of $\Dc_R$ -modules
  \begin{eqnarray*}
    \psi_{\alpha} :   \frac {\Dc_R}{\Dc_R \mf} & \to &  \frac {\Dc_R}{\Dc_R \mf^{n+1}}\\
    P  &\mapsto &  P  Q_{n-j,d}(x_1, \dots , x_d) X^{\alpha} \omod \Dc_R \mf^{n+1},
  \end{eqnarray*}
  and we put $\psi_\alpha (1_{\alpha})=m_{\alpha} \in (\frac
  {\Dc_R}{\Dc_R \mf^{n+1}})^{\mf} $, where $1_{\alpha}$ is the cyclic
  generator of the term with index $\alpha$ in the right side of (2).

  The fact that any differential operator $P\in \Dc_R$ has a unique
  expansion $P = \sum \partial^\alpha a_\alpha$, $a_\alpha \in R$,
  implies that $Q_{n-j,d}(x_1, \dots , x_d) X^{\alpha} \not \in \Dc_R
  \mf^{n+1}$, when $|\alpha | \leq j$; hence $\psi_\alpha \neq 0$.

  \Lemma{\ref{simplelemma}} implies that $\psi$ is injective if we
  first prove that the vectors $m_\alpha$ are linearly independent in
  the $k_R$-vector space $(\frac {\Dc_R}{\Dc_R
    \mf^{n+1}})^{\mf}$. Assume that we have a linear relation
  \begin{displaymath}
    \sum_{|\alpha| \leq n} \lambda_{\alpha} m_{\alpha} =0, \quad     \lambda_{\alpha}    \in k_A, 
  \end{displaymath}
which means 
\begin{displaymath}
  \sum_{|\alpha|\leq n} \hat \lambda_{\alpha} Q_{n-j}X^\alpha \in \Dc \mf^{n+1},
  \quad \hat \lambda_\alpha \in R,\quad 
  \hat \lambda_\alpha \omod \mf = \lambda_\alpha.
\end{displaymath}
Defining the Euler operator $\nabla = \sum_{i=1}^d x_i \partial_{x_i}$
we have  $[\nabla,  Q_{n-j}] =0$, $[\nabla, X^\alpha]=
|\alpha|X^\alpha$, and
\begin{displaymath}
  \hat \lambda_{\alpha} Q_{n-j}X^\alpha \nabla = 
 (d -  |\alpha|)\hat \lambda_\alpha    Q_{n-j}X^\alpha  - \nabla (\hat
  \lambda_\alpha)    Q_{n-j}X^\alpha  +   (\nabla-d) \hat \lambda_{\alpha} Q_{n-j}X^\alpha .
\end{displaymath}
Here the two last terms on the right belong to $\Dc_R \mf^{n+1}$ due
to \thetag{$*$}, after noting that $\nabla (\hat \lambda_\alpha) \in
\mf$ and $\nabla -d = \sum_{i=1}^d \partial_{x_i}x_i $. Therefore we can
define a $k_R$-linear action $E$ on the linear space
$\sum_{|\alpha|\leq n} k_R
m_\alpha\subset (\frac {\Dc_R}{\Dc_R \mf^{n+1}})^{\mf}$ such that $E
m_\alpha = (d-|\alpha|)m_\alpha$. A standard weight argument now implies
that all the coefficients $\lambda_\alpha =0$.

It remains to prove that $\psi$ is an isomorphism.  Let $N_n\subset
\Dc_R/\Dc_R \mf^{n+1}$ be the submodule that is generated by the
canonical projection of $ \mf^n $ in $\Dc_R/\Dc_R \mf^{n+1}$.  Then
$N_n$ is generated by its $\mf$-invariants and $\dim_{k_R} N_n ^\mf $
equals the number of monomials of degree $n$ in $d$ variables, which
is thus equal to the length of $N_n$ by
\Lemma{\ref{simplelemma}}. Since $\Dc_R/\Dc_R \mf^{n+1}/N_n =
\Dc_R/\Dc_R \mf^n$, an induction over $n$ gives that the lengths of
both sides in (2) are equal. Since $\psi$ is injective this implies
that $\psi$ is an isomorphism.
\end{pfof}

\begin{remark}
  \begin{enumerate}
  \item The proof gives
    \begin{displaymath}
      (\frac {\Dc_R}{\Dc_R \mf^{n+1}})^{\mf} =
      \sum_{|\alpha|\leq n} k_R Q_{n-j,d}(x_1, \dots , x_d) X^{\alpha}
      \mod \Dc_R \mf^{n+1}.
  \end{displaymath}

\item Given a resolution as in \thetag{\ref{equation}}, then (2) in
  \Theorem{\ref{decompDmod}} will give a decomposition of $M$.
  \end{enumerate}
\end{remark}
\begin{example} Let $A_1(k)$ be the Weyl algebra in one variable over
  a field $k$ of characteristic $0$, and consider a maximal ideal
  $\mf$ in the polynomial ring $ k[x]\subset A_1(k)$; let $\partial_x$
  be the $k$-linear derivation of $k[x]$ such that
  $\partial_x(x)=1$. By \Theorem{\ref{semisimplefinitesupport}}, $M_l=
  A_1(k)/A_1(k)\mf^{l+1}= A_1(k) m_l$ is a semi-simple $A_1(k)$-module
  (here $m_l= 1 \omod A_1(k) \mf^{l+1}$).  The localisation $R=
  k[x]_{\mf}$ has a regular system of parameters formed by a generator
  $x_1$ of the principal ideal $\mf$, and given $x_1$ there exists a
  unique derivation $\partial_{x_1}\in T_{R/k}$ such that
  $\partial_{x_1}(x_1) =1$. The $\Dc_{R/k}$-module $R\otimes_{k[x]}M_l$
  can be decomposed as follows:
  \begin{eqnarray*}
    &&    \bigoplus_{i=0}^l R\otimes_{k[x]} M_0  \to  R\otimes_{k[x]}M_l,\\
    (P_im_{0,i})& \mapsto &  P_0\cdot x_1^lm_l + P_1\cdot (1+ \partial_{x_1}\cdot x_1) x_1^{l-1}m_l + \cdots  \\  &+& P_{l}\cdot
    (1+ \partial_{x_1}\cdot x_1) (2 + \partial_{x_1}\cdot x_1)
    \cdots (l+ \partial_{x_1}\cdot x_1 ) m_l.
\end{eqnarray*} 
The isomorphism $ M_0 \oplus \cdots \oplus M_0 \to M_l$, where there
are $l+1$ terms on the left, is defined similarly. We notice that
although $\partial_{x_1}x_1 $ in general does not act on $k[x]$, it has a
well-defined action on $M_l$. Note also that the invariant space
$M_0^{\mf} $ of the simple module $M_0$ is 1-dimensional over the
residue field $k_R$ and thus its dimension over $k$ equals the degree
of the field extension $k_R/k$.

One can reverse the roles of $\partial_x$ and $x$ in the Weyl algebra
$A_1(k)$, and instead decompose modules according to their support in
$\Spec k[\partial_{x}]$.  Thus if $P\in k[\partial_x] \subset A_1(k)$
is a differential operator with constant coefficients, then
\begin{displaymath}
  \frac{A_1(k)}{A_1(k) P} \cong  \frac{A_1(k)}{A_1(k) P_1}\oplus \cdots  \oplus \frac{A_1(k)}{A_1(k) P_r}
\end{displaymath}
where $P= P_1 \cdots P_r $ is a factorisation into irreducible
polynomials, where repetitions may occur.  It is a good exercise to
write down an isomorphism for some concrete polynomial $P$ using
``Pochhammer'' operators $Q_{j,1}(P_i)$ when $P$ has multiple factors.
\end{example}

\end{document}